\pdfoutput=1
\documentclass[11pt,oneside]{amsart}
\usepackage[paperheight=279mm,paperwidth=18cm,textheight=26cm,textwidth=14cm]{geometry}
\usepackage[T1]{fontenc}
\usepackage[utf8]{inputenc}
\usepackage[unicode]{hyperref}
\hypersetup{
bookmarks=true,
colorlinks=true,
citecolor=[rgb]{0,0,0.5},
linkcolor=[rgb]{0,0,0.5},
urlcolor=[rgb]{0,0,0.75},
pdfpagemode=UseNone,
pdfstartview=FitH,
pdfdisplaydoctitle=true,
pdftitle={A1 for martingale maximal function},
pdfauthor={Pavel Zorin-Kranich},
pdflang=en-US
}

\usepackage{amsmath}
\usepackage{amssymb}
\usepackage{amsthm}
\usepackage{xspace}
\usepackage{mathtools}

\usepackage[style=alphabetic,maxalphanames=4]{biblatex}
\numberwithin{equation}{section}
\theoremstyle{plain}
\newtheorem{theorem}{Theorem}[section]
\newtheorem{proposition}[theorem]{Proposition}
\newtheorem{lemma}[theorem]{Lemma}
\newtheorem{corollary}[theorem]{Corollary}

\theoremstyle{remark}
\newtheorem{remark}[theorem]{Remark}

\DeclarePairedDelimiter\abs{\lvert}{\rvert}
\DeclarePairedDelimiter\norm{\lVert}{\rVert}

\def\E{\mathbb{E}}
\DeclarePairedDelimiterXPP\EE[1]{\E}{\lparen}{\rparen}{}{#1} 

\DeclarePairedDelimiterX\Set[1]\{\}{%

#1
}

\newcommand*{\N}{\mathbb{N}}
\newcommand*{\R}{\mathbb{R}}
\newcommand*{\calF}{\mathcal{F}}
\newcommand*{\frakf}{\mathfrak{f}}

\newcommand{\dif}{\mathop{}\!\mathrm{d}} 

\newcommand{\Csm}{C_{\mathrm{sm}}}
\newcommand{\CH}{C_{\mathrm{H}}}

\title[$A_{1}$ for martingale maximal function]{$A_{1}$ Fefferman--Stein inequality\\ for maximal functions of martingales\\ in uniformly smooth spaces}
\author[P.~Zorin-Kranich]{Pavel Zorin-Kranich}
\address{Mathematical Institute\\ University of Bonn}
\email{pzorin@uni-bonn.de}

\makeatletter
\@namedef{subjclassname@2020}{%
  \textup{2020} Mathematics Subject Classification}
\makeatother

\begin{document}
\maketitle
\begin{abstract}
Let $f$ be a martingale with values in a uniformly $p$-smooth Banach space and $w$ any positive weight.
We show that $\E (f^{*} \cdot w) \lesssim \E(S_{p}f \cdot w^{*})$, where $\cdot^{*}$ is the martingale maximal operator and $S_{p}$ is the $\ell^{p}$ sum of martingale increments.
\end{abstract}

\section{Introduction}
A Banach space $(X,\abs{\cdot})$ is called \emph{$(p,\Csm)$-smooth} (with $p\in [1,2]$ and $\Csm \in \R_{>0}$) if, for every $x,y\in X$, we have
\begin{equation}
\label{eq:pC-smooth}
\frac{1}{2} \Bigl(\abs{x+y}^{p} + \abs{x-y}^{p}\Bigr) \leq \abs{x}^{p} + \Csm^{p}\abs{y}^{p}.
\end{equation}
The most basic examples are that, for any $r\in (1,2]$, any $L^{r}$ space is $(r,1)$-smooth, see \cite[(10.33)]{MR3617459} (this is also a consequence of Clarkson's inequality), and, for any $r \in [2,\infty)$, any $L^{r}$ space is $(2,r-1)$-smooth, this follows from \cite[(10.37)]{MR3617459} and Jensen's inequality.
In general, unless $X$ is zero-dimensional, we must have $\Csm\geq 1$, as can be seen by taking $x=0$ in \eqref{eq:pC-smooth}.

Our main result is the following.
\begin{theorem}
\label{thm:w-BDG}
Let $p \in (1,2]$.
Let $(f_{n})_{n\in\N}$ be a martingale on a filtered probability space $(\Omega,(\calF_{n})_{n})$ with values in a $(p,\Csm)$-smooth Banach space $X$ and $w : \Omega \to \R_{\geq 0}$ a measurable function (called a \emph{weight}).
Then,
\begin{equation}
\label{eq:w-BDG}
\E (f^{*} w)
\leq
84p' \Csm \E (S_{p}f \cdot w^{*}),
\end{equation}
where $p'$ denotes the Hölder conjugate $1/p'+1/p=1$, and
\[
S_{p}f = \bigl( \abs{f_{0}}^{p} + \sum_{n=1}^{\infty} \abs{f_{n}-f_{n-1}}^{p} \bigr)^{1/p},
\quad
f^{*} = \sup_{n\in\N} \abs{f_{n}},
\quad
w^{*} = \sup_{n\in\N} \E(w|\calF_{n}).
\]
\end{theorem}
In order to put Theorem~\ref{thm:w-BDG} into context, we list the previously known cases (in each of which the inequality \eqref{eq:w-BDG} is in fact known with a smaller constant).
\begin{enumerate}
\item The unweighted ($w=1$) scalar ($X=\R$) case is one of the Burkholder--Davis--Gundy inequalities \cite{MR0268966}.
\item The scalar ($X=\R$) case, which served as the main inspiration for this work, was proved in \cite{MR3688518}.
\item The unweighted ($w=1$) case is one of the implications in the characterization of martingale type, see \cite[Theorem 10.60]{MR3617459}.
\end{enumerate}

We follow \cite{MR3688518} in calling the inequality \eqref{eq:w-BDG} a Fefferman--Stein inequality, in reference to \cite[\textsection 3]{MR0284802}, where the first inequality involving the pair of weights $w,w^{*}$ appeared (see \cite[Theorem 3.2.3]{MR3617205} for a martingale version).
In order to distinguish this result from many others due to Fefferman and Stein, we prepend the designation ``$A_{1}$'', which in the one-weight theory stands for the condition $w^{*} \leq [w]_{A_{1}}w$.
The pair $w,w^{*}$ can be seen as satisfying a two-weight version of the $A_{1}$ condition.

For dyadic martingales, assuming $w\in A_{\infty}$, an inequality similar to \eqref{eq:w-BDG} with $w^{*}$ replaced by $w$ is known \cite[Theorem 2]{MR352854}.
The recent result \cite[Theorem 1.3]{BrzOse2021} (which applied to martingale transforms in place of the square function) suggests that no such inequality is possible for general martingales.

The advantage of weighted estimates such as \eqref{eq:w-BDG} is that they can be easily extrapolated to estimates for other moments, see Appendix~\ref{sec:extrapolation}.
We illustrate the extrapolation idea with a basic argument, which shows that the linear dependence on $\Csm$ in \eqref{eq:w-BDG} is optimal.
Assume that the inequality
\[
\E (f^{*} w)
\leq
K \E (S_{p}f \cdot w^{*})
\]
holds for all weights $w$.
By Hölder's inequality and Doob's maximal inequality, see e.g.\ \cite[Theorem 3.2.2]{MR3617205}, for any $r\in (1,\infty)$, we obtain
\[
\E (f^{*} w)
\leq
K \norm{S_{p}f}_{L^{r}} \norm{w^{*}}_{L^{r'}}
\leq
Kr \norm{S_{p}f}_{L^{r}} \norm{w}_{L^{r'}}.
\]
Since $L^{r'}$ is the dual space of $L^{r}$ , this implies
\begin{equation}
\label{eq:BDG-Lr}
\norm{ f^{*} }_{L^{r}}
\leq
Kr \norm{S_{p}f}_{L^{r}}.
\end{equation}
Incidentally, the linear growth in $r$ of the constant in the inequality \eqref{eq:BDG-Lr} is optimal in the scalar case $X=\R$, $p=2$, see \cite[Theorem 3.2]{MR365692}.

Let now $X$ be a Banach space such that the inequality \eqref{eq:BDG-Lr} holds with $p=r \in (1,2]$ for all martingales $f$ with values in $X$.
By Pisier's renorming theorem \cite[Theorem 10.22]{MR3617459}, the space $X$ admits an equivalent norm that is $(p,CK)$-smooth for some $C$ depending only on $p$.
In this sense, the linear dependence of \eqref{eq:w-BDG} on $\Csm$ is optimal.

The dependence of the bound \eqref{eq:w-BDG} on $p'$ does not seem natural, since it does not appear in the corresponding non-maximal bound \eqref{eq:w-non-maximal}.
Also, the $p=1$ bound clearly holds with constant $1$.
Therefore, we find it reasonable to conjecture that $84p'$ in \eqref{eq:w-BDG} can be replaced by a constant that does not depend on $p$.

\subsection{Non-martingale version}
The proof of Theorem~\ref{thm:w-BDG} in fact yields a more general statement, involving processes with a structure that was introduced in \cite[Theorem 3.1]{arxiv:2006.06964}.
Let $(\Omega,(\calF_{n})_{n\in\N})$ be a filtered probability space and $(g_{n})_{n\in\N}$, $(f_{n})_{n\in\N}$, $(\tilde{f}_{n})_{n\in\N}$ be adapted processes with values in a $(p,\Csm)$-smooth Banach space $X$.
Assume that $f_{0}=\tilde{f}_{0}=0$, and for every $n\in\N_{>0}$ we have
\[
f_{n} = \tilde{f}_{n-1} + (g_{n}-g_{n-1}),
\quad
\abs{\tilde{f}_{n}} \leq \abs{f_{n}}.
\]
Then,
\begin{equation}
\label{eq:w-BDG:non-martingale}
\E (f^{*} w)
\leq
84p' \Csm \E (S_{p}g \cdot w^{*}).
\end{equation}
As in \eqref{eq:BDG-Lr}, for $r\in [1,\infty)$, this implies
\begin{equation}
\label{eq:BDG-Lr:non-martingale}
\norm{f^{*}}_{L^{r}}
\leq
84p' \Csm r \norm{ S_{p}g }_{L^{r}}.
\end{equation}
The Rosenthal-type inequality in \cite[Theorem 3.1]{arxiv:2006.06964} states that, if $X$ is a $(2,\Csm)$-space, then
\begin{equation}
\label{eq:Rosenthal}
\norm{f^{*}}_{L^{r}}
\leq
30 r \norm{\sup_{n} \abs{g_{n}-g_{n-1}} }_{L^{r}}
+ 40 \Csm r^{1/2} \norm{ sg }_{L^{r}},
\end{equation}
where $sg$ is the conditional square function:
\[
sg = \Bigl( \sum_{n} \E(\abs{g_{n}-g_{n-1}}^{2}|\calF_{n-1}) \Bigr)^{1/2}.
\]
For $r\geq 2$, \eqref{eq:Rosenthal} implies \eqref{eq:BDG-Lr:non-martingale}, since
\[
\norm{sg}_{L^{r}} \leq (r/2)^{1/2} \norm{S_{2}g}_{L^{r}},
\quad r\in [2,\infty),
\]
by Doob's maximal inequality and duality.
On the other hand, the version of \eqref{eq:Rosenthal} for $r<2$ in \cite[Corollary 3.6]{arxiv:2006.06964} is not obviously related to \eqref{eq:BDG-Lr:non-martingale}.

\subsection{Outline of the article}
The proof of Theorem~\ref{thm:w-BDG} is based on the Bellman function technique; we refer to the books \cite{MR2964297,VasVol_Bellman_book} for other instances of this technique.

In Section~\ref{sec:unif-smooth}, we review the characterization of uniform smoothness that will be used in the proofs of our main results.

In Section~\ref{sec:Bellman-4}, we prove the inequality \eqref{eq:w-non-maximal}, which is a non-maximal version of Theorem~\ref{thm:w-BDG}.
The proof of that inequality uses a Bellman function that is adapted from \cite{MR3688518}.
Although that inequality will not be used in the proof of Theorem~\ref{thm:w-BDG}, the Bellman function estimate in Proposition~\ref{prop:concavity0} will be used again there.

In Section~\ref{sec:Bellman-5}, we prove the full Theorem~\ref{thm:w-BDG}.
This is accomplished using a Bellman function that combines features present in the articles \cite{MR3322322} and \cite{MR3688518}.

In Appendix~\ref{sec:extrapolation}, we give a sample application of the weighted bound \eqref{eq:w-BDG}.

\section{General facts about uniformly smooth spaces}
\label{sec:unif-smooth}
We will use the regularity properties of the norm on a uniformly smooth Banach space that can be found e.g.\ in \cite[Lemma 2.1]{MR2861433}.
We take the opportunity to streamline the deduction of these properties from \eqref{eq:pC-smooth}.
The following lemma is a minor variant of \cite[Lemma I.1.3]{MR1211634} (there, the case $\phi(x)=\abs{x}$ is considered).
\begin{lemma}
\label{lem:Frechet-dif}
Let $(X,\abs{\cdot})$ be a Banach space, $\phi : X \to \R$ a convex function, and $x \in X$ such that
\begin{equation}
\label{eq:Frechet-dif:Lipschitz}
L := \limsup_{y \to 0} \frac{\abs{\phi(x+y)-\phi(x)}}{\abs{y}} < \infty,
\quad\text{and}
\end{equation}
\begin{equation}
\label{eq:Frechet-dif:uniform}
\lim_{y\to 0} \frac{\phi(x+y)+\phi(x-y)-2\phi(x)}{\abs{y}} = 0.
\end{equation}
Then $\phi$ is Fr\'echet differentiable at $x$, and its derivative satisfies $\abs{\phi'(x)}_{X'} \leq L$.
\end{lemma}
\begin{proof}
Convexity implies that, for any $y\in X$, the function $t \mapsto \frac{\phi(x+ty)-\phi(x)}{t}$ is monotonically increasing in $t\in (0,\infty)$.
Therefore, there exist one-sided directional derivatives
\[
A(y) := \lim_{t\to 0+} \frac{\phi(x+ty)-\phi(x)}{t},
\]
and $\abs{A(y)}\leq L\abs{y}$ by \eqref{eq:Frechet-dif:Lipschitz}.
We will show that $A$ is the Fr\'echet derivative of $\phi$ at $x$.
From \eqref{eq:Frechet-dif:uniform}, it follows that $A(y)+A(-y)=0$ for all $y\in X$.
Hence, again by \eqref{eq:Frechet-dif:uniform}, we obtain
\begin{align*}
\MoveEqLeft
\sup_{\abs{y}\leq 1} \frac{\phi(x+ty)-\phi(x)}{t} - A(y)
\\ &\leq
\sup_{\abs{y}\leq 1} \frac{\phi(x+ty)-\phi(x)}{t} - A(y) + \frac{\phi(x-ty)-\phi(x)}{t} - A(-y)
\\ &=
\sup_{\abs{y}\leq 1} \frac{\phi(x+ty)+\phi(x-ty)-2\phi(x)}{t}
\xrightarrow{t\to 0}
0.
\end{align*}
This shows that the difference quotients of $\phi$ converge to $A$ locally uniformly.
It remains to show that $A$ is linear.
To this end, we first observe that $A$ is convex, since it is the limit of the convex functions $y \mapsto (\phi(x+ty)-\phi(x))/t$.
Then also $y\mapsto -A(y)=A(-y)$ is convex, so that $A$ is concave.
It follows that $A$ is affine.
Finally, $A(0)=0$.
\end{proof}

Let $(X,\abs{\cdot})$ be a $(p,\Csm)$-smooth Banach space with $p \in (1,2]$ and
\[
\phi(x) := \abs{x}^{p},
\quad x\in X.
\]
The hypothesis \eqref{eq:Frechet-dif:uniform} of Lemma~\ref{lem:Frechet-dif} follows directly from the definition \eqref{eq:pC-smooth}.
It is also easy to see that, for any $x\in X$, the hypothesis \eqref{eq:Frechet-dif:Lipschitz} holds with $L=L(x)=p\abs{x}^{p-1}$.
Therefore, Lemma~\ref{lem:Frechet-dif} implies that the function $\phi$ is Fr\'echet differentiable, and $\abs{\phi'(x)}_{X'} \leq p \abs{x}^{p-1}$.

Let $\CH \in [0,\infty]$ be the smallest constant such that, for any $x,y\in X$, we have
\begin{equation}
\label{eq:phi'-Holder}
\abs{\phi'(x)-\phi'(y)}_{X'} \leq \CH^{p} \abs{x-y}^{p-1}.
\end{equation}
The proof of \cite[Lemma V.3.5]{MR1211634} (with $\alpha=p-1$) shows that
\begin{equation}
\label{eq:CH<Csm}
\CH^{p} \leq 2^{p} \Csm^{p}.
\end{equation}
Conversely, for any $x,y\in X$, we have
\begin{align*}
\frac{1}{2} \Bigl( \abs{x+y}^{p} + \abs{x-y}^{p} \Bigr)
&=
\frac{1}{2} \Bigl( \phi(x) + \int_{0}^{1} \phi'(x+ty) y \dif t + \phi(x) + \int_{0}^{1} \phi'(x-ty) (-y) \dif t \Bigr)
\\ &\leq
\phi(x) + \frac{1}{2} \int_{0}^{1} \abs{\phi'(x+ty) - \phi'(x-ty)}_{X'} \abs{y} \dif t
\\ &\leq
\phi(x) + \frac{1}{2} \int_{0}^{1} \CH^{p} \abs{2ty}^{p-1} \abs{y} \dif t
=
\abs{x}^{p} + \frac{2^{p-2} \CH^{p}}{p} \abs{y}^{p}.
\end{align*}
Therefore, $\Csm^{p} \leq 2^{p-2}\CH^{p}/p$, so the conditions \eqref{eq:phi'-Holder} and \eqref{eq:pC-smooth} are equivalent.
However, we find the condition \eqref{eq:phi'-Holder} more convenient to use, so all subsequent results will be formulated in terms of $\CH$.
We note that $\CH^{p} \geq p$, as can be seen by considering a one-dimensional subspace of $X$.

\section{Bellman function for the martingale}
\label{sec:Bellman-4}
In this section, we adapt the Bellman function from \cite{MR3688518} to our setting.
This will allow us to prove the inequality
\begin{equation}
\label{eq:w-non-maximal}
\E (\abs{f} w)
\leq
9 \CH \E ( S_{p}f \cdot w^{*}).
\end{equation}
Note that, unlike in \eqref{eq:w-BDG}, the constant on the right-hand side of \eqref{eq:w-non-maximal} does not explicitly depend on $p$.

For $x\in X$, $q\geq 0$, and $0 \leq u \leq v$, let
\begin{equation}
\label{eq:Bellman-4}
U(x,q,u,v) := u (\abs{x}^p/\CH^{p}+q)^{1/p} - Cvq^{1/p} + \tilde{C} vq^{1/p} \ln(1+u/v).
\end{equation}
We denote the $x$- and the $u$-derivatives of $U$ by $U_{x}$ and $U_{u}$, respectively.
Note that $U$ is indeed Fr\'echet differentiable in $x$, and the derivative is given by
\[
U_{x}(x,q,u,v)h = \frac{\phi'(x)h}{p\CH^{p}(\abs{x}^{p}/\CH^{p}+q)^{1-1/p}}.
\]
The main feature of the function \eqref{eq:Bellman-4} is the following concavity property.,

\begin{proposition}
\label{prop:concavity0}
Suppose that $C=9$ and $\tilde{C} = 4\sqrt{2}$.
Then, for any $x,d\in X$, $q,u,v\in \R_{\geq 0}$, and $e\in\R$ with $u \leq v$ and $0 \leq u+e$, we have
\begin{equation}
\label{eq:concavity0}
U(x+d,q+\abs{d}^{p},u+e,(u+e)\vee v)
\leq
U(x,q,u,v) + U_{x}(x,q,u,v)d + U_{u}(x,q,u,v)e.
\end{equation}
\end{proposition}

Before turning to the verification of \ref{eq:concavity0}, let us quickly show why it is useful.

\begin{proof}[Proof of \eqref{eq:w-non-maximal} assuming Proposition~\ref{prop:concavity0}]
Let $w_{n} := \E(w|\calF_{n})$ and $w^{*}_{n} := \max_{n'\leq n} w_{n'}$.
For each $n$, we apply Proposition~\ref{prop:concavity0} with
\begin{multline}
\label{eq:Bellman-parameters}
x = f_{n}, \quad
q = q_{n} = \abs{f_{0}}^{p} + \sum_{m=1}^{n} \abs{f_{n}-f_{n-1}}^{p}, \\
u = w_{n}, \quad
v = w^{*}_{n}, \quad
d = f_{n+1}-f_{n}, \quad
e = w_{n+1}-w_{n}.
\end{multline}
Taking the conditional expectation on both sides of the resulting inequality, we obtain
\[
\E U(f_{n+1},q_{n+1},w_{n+1},w^{*}_{n+1})
\leq
\E U(f_{n},q_{n},w_{n},w^{*}_{n}).
\]
Iterating this inequality, we obtain
\[
\E (w_{N} \abs{f_{N}}/\CH - C w^{*}_{N} q_{N}^{1/p})
\leq
\E U(f_{N},q_{N},w_{N},w^{*}_{N})
\leq
\E U(f_{0},q_{0},w_{0},w^{*}_{0})
\leq 0.
\]
This implies \eqref{eq:w-non-maximal}.
\end{proof}

Unlike in the scalar case in \cite{MR3688518}, it does not seem possible to directly use the Bellman function \eqref{eq:Bellman-4} to deduce the maximal estimate \eqref{eq:w-BDG}.
However, Proposition~\ref{prop:concavity0} will be used in the proof of Proposition~\ref{prop:concavity}, which will in turn imply the maximal estimate.

We did not attempt to optimize the numerical values of $C,\tilde{C}$ in Proposition~\ref{prop:concavity0}.
Also the conditions \eqref{eq:cond0-tC} and \eqref{eq:cond0-C}, according to which these values are chosen, can be improved by a more careful choice of numerical constants at various places in the proof.
However, we should like to point out that the main loss compared to \cite{MR3688518} is due to the use of the estimate \eqref{eq:sloppy} in several denominators.

\begin{proof}[Proof of Proposition~\ref{prop:concavity0}]
The inequality \eqref{eq:concavity0} is quite delicate for small values of $d$ and $e$, and quite sloppy for large values.
This can be seen by looking at the asymptotic behavior of \eqref{eq:concavity0} for $d\to\infty$ or $e\to\infty$, which is dominated by the term $-C((u+e)\vee v)(q+\abs{d}^{p})^{1/p}$.
Accordingly, we distinguish the following cases.
\begin{enumerate}
\item $\abs{d}^{p} \leq q/2$ and $u+e \leq v$,
\item $\abs{d}^{p} \geq q/2$, $u+e\leq v$,
\item $u+e \geq v$.
\end{enumerate}

Throughout the proof, let
\begin{equation}
\label{eq:def:psi}
\psi(t) := \abs{x+td}^{p}/\CH^{p},
\quad a:=\psi(0),
\quad b:=\psi(0).
\end{equation}
As a consequence of Lemma~\ref{lem:Frechet-dif} and \eqref{eq:phi'-Holder}, we have
\begin{equation}
\label{eq:psi':size}
\abs{\psi'(t)}
\leq
p \abs{x+td}^{p-1} \abs{d}/\CH^{p}
=
p (\psi(t))^{1-1/p} \abs{d}/\CH
\leq
p (\psi(t))^{1-1/p} \abs{d},
\end{equation}
\begin{equation}
\label{eq:psi':H}
\abs{\psi'(t)-\psi'(\tilde{t})} \leq \abs{t-\tilde{t}}^{p-1} \abs{d}^{p}.
\end{equation}

\textbf{Case 1.}
Suppose that $u+e \leq v$ and $\abs{d}^{p} \leq q/2$.

We have
\begin{equation}
\label{eq:psi-linear-approx}
\abs{\psi(t) - \psi(0) - t\psi'(0)}
\leq
\int_{0}^{t} \abs{\psi'(\tilde{t})-\psi'(0)} \dif\tilde{t}
\leq
\int_{0}^{t} \tilde{t}^{p-1} \abs{d}^{p} \dif\tilde{t}
=
\frac{1}{p} \abs{td}^{p}.
\end{equation}
By the AMGM inequality, we have
\begin{align*}
\abs{\psi'(0)}
&\leq
p \psi(0)^{1-1/p} \abs{d}
\leq
\frac{1}{2} \psi(0) + p^{p-1} \abs{d}^{p}
\leq
\frac{1}{2} \psi(0) + q,
\\
\abs{\psi'(0)}
&\leq
p \psi(0)^{1-1/p} \abs{d}
\leq
\psi(0) + \abs{d}^{p}
\leq
\psi(0) + q/2.
\end{align*}
This implies in particular that, for any $t\in [0,1]$, we have
\begin{equation}
\label{eq:sloppy}
\psi(0) + \psi'(0)t + q \geq \max(\psi(0)/2,q/2).
\end{equation}

Let
\begin{align*}
G(t) &:=
U(x+td,q+\abs{td}^{p},u+te,v)
\\ &=
(u+te)(\psi(t)+q+\abs{td}^{p})^{1/p} - v (q+\abs{td}^{p})^{1/p} (C - \tilde{C} \ln(1+(u+te)/v)).
\end{align*}
The claim is then equivalent to $G(1) \leq G(0) + G'(0)$.
Let also
\[
H(t) := (u+te) (\psi(0)+\psi'(0)t+q)^{1/p}
- C_{5} v (q+\abs{td}^{p})^{1/p}
- vq^{1/p} (C_{6} - \tilde{C} \ln(1+(u+te)/v)),
\]
where the splitting $C=C_{5}+C_{6}$ will be chosen later.
Then $H(0) = G(0)$ and $H'(0) = G'(0)$.
In view of \eqref{eq:psi-linear-approx}, we have
\begin{equation}
\label{eq:1}
\psi(1)+q+\abs{d}^{p}
\leq
\psi(0) + \psi'(0) + \frac{1}{p} \abs{d}^{p} + q + \abs{d}^{p},
\end{equation}
and it follows that
\begin{multline*}
G(1) - H(1)
=
(u+e)(\psi(1)+q+\abs{d}^{p})^{1/p} - v (q+\abs{d}^{p})^{1/p} (C - \tilde{C} \ln(1+(u+e)/v))
\\ -
\Bigl( (u+e) (\psi(0)+\psi'(0)+q)^{1/p}
- C_{5} v (q+\abs{d}^{p})^{1/p}
- vq^{1/p} (C_{6} - \tilde{C} \ln(1+(u+e)/v)) \Bigr)
\\ \leq
(u+e)(\psi(0)+\psi'(0)+q+(1+1/p)\abs{d}^{p})^{1/p} - v (q+\abs{d}^{p})^{1/p} (C_{6} - \tilde{C} \ln(1+(u+e)/v))
\\ -
\Bigl( (u+e) (\psi(0)+\psi'(0)+q)^{1/p}
- vq^{1/p} (C_{6} - \tilde{C} \ln(1+(u+e)/v)) \Bigr)
\\ = K(\abs{d}^{p}) - K(0),
\end{multline*}
where
\[
K(s) = (u+e)(\psi(0)+\psi'(0)+q+(1+1/p)s)^{1/p}
- v(q+s)^{1/p} (C_{6} - \tilde{C} \ln(1+(u+e)/v)).
\]
We have
\begin{align*}
K'(s)
&=
\frac{(u+e)(1+1/p)}{p(\psi(0)+\psi'(0)+q+(1+1/p)s)^{1-1/p}}
\\ &\quad
- \frac{v}{p(q+s)^{1-1/p}} (C_{6} - \tilde{C} \ln(1+(u+e)/v))
\\ &\leq
\frac{v(1+1/p)}{p(q/2+s/2)^{1-1/p}} - \frac{v}{p(q+s)^{1-1/p}} (C_{6} - \tilde{C} \ln(2))
\leq 0
\end{align*}
provided that
\begin{equation}
\label{eq:cond1}
C_{6} \geq 2^{1-1/p}(1+1/p) + \tilde{C} \ln(2).
\end{equation}

Next, to show that $H(1) \leq H(0) + H'(0)$, we show that $H'(t) \leq H'(0)$ for $t\in [0,1]$.
We compute
\begin{multline*}
H'(t) = e(a+bt+q)^{1/p} + \frac{(u+te) b}{p(a+bt+q)^{1-1/p}}
\\- \frac{C_{5} v t^{p-1} \abs{d}^{p}}{(q+\abs{td}^{p})^{1-1/p}}
+ e q^{1/p} \tilde{C} /(1+(u+te)/v),
\end{multline*}
\begin{multline*}
H''(t) = \frac{2 e b}{p(a+bt+q)^{1-1/p}}
- \frac{(1-1/p)(u+te) b^{2}}{p(a+bt+q)^{2-1/p}}
\\- \frac{(p-1) C_{5} v t^{p-2} \abs{d}^{p}}{(q+\abs{td}^{p})^{1-1/p}}
+ \frac{p(1-1/p) C_{5} v t^{2p-2} \abs{d}^{2p}}{(q+\abs{td}^{p})^{2-1/p}}
- e^{2} q^{1/p} \tilde{C} /(1+(u+te)/v)^{2}/v
\\ \leq
\frac{2 \abs{e} a^{1-1/p} \abs{d}}{(a+bt+q)^{1-1/p}}
- \frac{(p-1) C_{5} v t^{p-2} \abs{d}^{p} q}{(q+\abs{td}^{p})^{2-1/p}}
- e^{2} q^{1/p} \tilde{C} /(4v)
\\ \leq
2^{2-1/p} \abs{e} \abs{d}
- \frac{(p-1)C_{5} v t^{p-2} \abs{d}^{p}}{(3/2)^{2-1/p} q^{1-1/p}}
- e^{2} q^{1/p} \tilde{C} /(4v).
\end{multline*}
Integrating this inequality, we obtain
\[
H'(t)-H'(0)
=
\int_{0}^{t} H''(\tilde{t}) \dif \tilde{t}
\leq
2^{2-1/p} \abs{e} \abs{d} t
- \frac{C_{5} v t^{p-1} \abs{d}^{p}}{(3/2)^{2-1/p} q^{1-1/p}}
- e^{2} q^{1/p} \tilde{C} /(4v) t.
\]
By the AMGM inequality,
\begin{equation}
\label{eq:AMGM-ed}
\begin{split}
\abs{e} \abs{d}
&=
\bigl( \frac{v \abs{d}^{p}}{q^{1-1/p}} \bigr)^{1/p}
\bigl( \frac{\abs{e}^{p'} q^{1/p}}{v^{p'/p}} \bigr)^{1-1/p}
\\ &\leq
\bigl( \frac{v \abs{d}^{p}}{q^{1-1/p}} \bigr)^{1/p}
\bigl( \frac{\abs{e}^{2} q^{1/p}}{v} \bigr)^{1-1/p}
\\ &\leq
\frac{1}{p} \frac{v \abs{d}^{p}}{q^{1-1/p}}
+ \frac{1}{p'} \frac{\abs{e}^{2} q^{1/p}}{v}
\end{split}
\end{equation}
Hence, $H'(t) \leq H'(0)$ provided that
\begin{equation}
\label{eq:cond2}
C_{5} \geq 3^{2-1/p}/p,
\quad
\tilde{C} \geq 2^{4-1/p}/p'.
\end{equation}

\textbf{Case 2.}
Suppose now $\abs{d}^{p} \geq q/2$, $u+e\leq v$.
Let
\[
I(t) := U(x+td,q+\abs{td}^{p},u+e,v) - U(x,q,u,v) - U_{x}(x,q,u,v) td - U_{u}(x,q,u,v)e.
\]
For $\abs{td}^{p} \leq q/2$, we showed $I(t) \leq 0$ in the previous step.
Hence, it suffices to show $I'(t) \leq 0$ for all $t\in [0,1]$ such that $\abs{td}^{p} \geq q/2$.
We have
\begin{multline*}
I'(t)
= \frac{(u+e)(\psi'(t)+pt^{p-1}\abs{d}^{p})}{p(\psi(t) + q + \abs{td}^{p})^{1-1/p}}
\\- \frac{v t^{p-1}\abs{d}^{p}}{(q+\abs{td}^{p})^{1-1/p}} (C-\tilde{C} \ln(1+(u+e)/v))
- \frac{u \psi'(0)}{p(\psi(0)+q)^{1-1/p}}
\\ \leq
\frac{v \abs{d} (\psi(t)^{1-1/p} + \abs{td}^{p-1})}{(\psi(t) + q + \abs{td}^{p})^{1-1/p}}
- \frac{v t^{p-1}\abs{d}^{p}}{(q+\abs{td}^{p})^{1-1/p}} (C-\tilde{C} \ln(2))
+ \frac{v \psi(0)^{1-1/p} \abs{d}}{(\psi(0)+q)^{1-1/p}}
\\ \leq
2^{1/p} v \abs{d} - \frac{v \abs{d}}{3^{1-1/p}} (C-\tilde{C}\ln 2) + v \abs{d} \leq 0
\end{multline*}
provided that
\begin{equation}
\label{eq:cond3}
C \geq 3^{1-1/p} (1+2^{1/p}) + \tilde{C} \ln 2.
\end{equation}

\textbf{Case 3.}
Suppose now that $u+e \geq v$.
We want to show
\[
J(e) := U(x+d,q+\abs{d}^{p},u+e,u+e) - U(x,q,u,v) - U_{x}(x,q,u,v) d - U_{u}(x,q,u,v)e \leq 0.
\]
For $u+e = v$, we have shown that $J(e) \leq 0$ in the previous steps.
Hence, it suffices to show $J'(e) \leq 0$ for $e\geq v-u$.
We have
\begin{multline*}
J'(e)
= (\abs{x+d}^{p}/\CH^{p}+q+\abs{d}^{p})^{1/p} - (q+\abs{d}^{p})^{1/p} (C-\tilde{C}\ln 2)
\\- (\abs{x}^{p}/\CH^{p}+q)^{1/p} - q^{1/p} \tilde{C}/(1+u/v)
\\ \leq
\abs{x+d}/\CH + q^{1/p} + \abs{d}
- (q+\abs{d}^{p})^{1/p} (C-\tilde{C}\ln 2) - \abs{x}/\CH - q^{1/p} \tilde{C}/2
\\ \leq
2\abs{d}
- \abs{d} (C-\tilde{C}\ln 2) - q^{1/p} (\tilde{C}/2-1)
\leq 0
\end{multline*}
provided that
\begin{equation}
\label{eq:cond4}
\tilde{C} \geq 2,
\quad
C \geq 2 + \tilde{C}\ln 2.
\end{equation}
The inequalities \eqref{eq:cond1}, \eqref{eq:cond2}, \eqref{eq:cond3}, \eqref{eq:cond4} can be summarized as
\begin{align}
\label{eq:cond0-tC}
\tilde{C} &\geq \max(2,2^{4-1/p}/p'),
\\ \label{eq:cond0-C}
C &\geq \max(2,3^{1-1/p} (1+2^{1/p}),3^{2-1/p}/p+2^{1-1/p}(1+1/p)) + \tilde{C} \ln 2.
\end{align}
Plotting these functions, we see that the inequalities are satisfied with the claimed values of $C,\tilde{C}$.
\end{proof}

\section{Bellman function for the maximal function}
\label{sec:Bellman-5}
In this section, we combine the Bellman functions from \cite{MR3688518} and \cite{MR3322322}.
For $x\in X$, $\abs{x} \leq m$, $q\geq 0$, and $0 \leq u \leq v$, let
\begin{equation}
\label{eq:Bellman-fct}
\begin{split}
\MoveEqLeft
U(x,m,q,u,v)
\\ &:=
u (m^{p}/\CH^{p}+q)^{1/p}-\frac{u}{p}\frac{m^{p}/\CH^{p}-\abs{x}^{p}/\CH^{p}}{(m^{p}/\CH^{p}+q)^{1-1/p}}
-Cvq^{1/p} + \tilde{C} vq^{1/p} \ln(1+u/v)
\\ &=
\frac{u}{p'} (m^{p}/\CH^{p}+q)^{1/p} + \frac{u}{p}\frac{q+\abs{x}^{p}/\CH^{p}}{(m^{p}/\CH^{p}+q)^{1-1/p}}
-Cvq^{1/p}+ \tilde{C} vq^{1/p} \ln(1+u/v)
\end{split}
\end{equation}

Evidently, the function \eqref{eq:Bellman-fct} is a modification of \eqref{eq:Bellman-4}.
The most obvious such modification would be to replace $\abs{x}$ by $m$; the more sophisticated modification in \eqref{eq:Bellman-fct} is chosen in such a way that the left-hand side of \eqref{eq:concavity} becomes differentiable in $d$.
The following concavity property is the main feature of the function \eqref{eq:Bellman-fct}.

\begin{proposition}
\label{prop:concavity}
Let $p\in (1,2]$, $C=21$, and $\tilde{C}=4\sqrt{2}$.
Then, for any $x,d\in X$, $m,q,u,v\in \R_{\geq 0}$, and $e\in\R$ with $|x| \leq m$, $u\leq v$, and $u+e \geq 0$, we have
\begin{multline}
\label{eq:concavity}
U\bigl(x+d, m \vee|x+d| , q+\abs{d}^{p}, u+e, (u+e) \vee v \bigr)
\\ \leq
U(x,m,q,u,v)
+
U_{x}(x,m,q,u,v)d
+
U_{u}(x,m,q,u,v)e.
\end{multline}
\end{proposition}

The numerical value of $C$, which comes out of the condition \eqref{eq:cond-C}, is again probably far from optimal.

\begin{proof}[Proof of Theorem~\ref{thm:w-BDG} assuming Proposition~\ref{prop:concavity}]
We apply Proposition~\ref{prop:concavity} with the same parameters as in \eqref{eq:Bellman-parameters}, and additionally
\[
m = f^{*}_{n} := \max_{n'\leq n} \abs{f_{n'}}.
\]
Taking the conditional expectation on both sides of the resulting inequality, we obtain
\[
\E U(f_{n+1},f^{*}_{n+1},q_{n+1},w_{n+1},w^{*}_{n+1})
\leq
\E U(f_{n},f^{*}_{n},q_{n},w_{n},w^{*}_{n}).
\]
Iterating this inequality, we obtain
\[
\E \bigl( \frac{w_{N} f^{*}_{N}}{p' \CH} - C w^{*}_{N} q_{N}^{1/p})
\leq
\E U(f_{N},f^{*}_{N},q_{N},w_{N},w^{*}_{N})
\leq
\E U(f_{0},f^{*}_{0},q_{0},w_{0},w^{*}_{0})
\leq 0.
\]
This implies
\begin{equation}
\label{eq:w-BDG:CH}
\E (f^{*} w)
\leq
21p' \CH \E (S_{p}f \cdot w^{*}),
\end{equation}
which in turn implies \eqref{eq:w-BDG} in view of \eqref{eq:CH<Csm}.

A similar argument also shows \eqref{eq:w-BDG:non-martingale}.
\end{proof}

\begin{remark}
Proposition~\ref{prop:concavity} can also be used to recover a non-maximal bound similar to \eqref{eq:w-non-maximal} (but with a larger absolute constant).
This is because, by the AMGM inequality,
\[
\frac{w_{N} \abs{f_{N}}}{\CH}
\leq
w_{N} ( (\abs{f_{N}}/\CH)^{p} + q_{N})^{1/p}
\leq
\frac{w_{N}}{p'} ((f^{*}_{N}/\CH)^{p}+q_{N})^{1/p} + \frac{w_{N}}{p}\frac{q_{N}+\abs{f_{N}}^{p}/\CH^{p}}{((f^{*}_{N}/\CH)^{p}+q_{N})^{1-1/p}}.
\]
\end{remark}

\begin{proof}[Proof of Proposition~\ref{prop:concavity}]
Due to an additional maximum in \eqref{eq:concavity}, we have to distinguish a few more cases than in Section~\ref{sec:Bellman-4}.
The main distinction is according to the ordering of $\abs{x+d}$ and $m$, since this ordering substantially affects the shape of the function \eqref{eq:Bellman-fct}.
The cases are as follows.
\begin{enumerate}
\item $\abs{x+d} \leq m$
\begin{enumerate}
\item $\abs{d}^{p} \leq q/2$, $u+e\leq v$,
\item $\abs{d}^{p} \leq q/2$, $u+e\geq v$,
\item $\abs{d}^{p} \geq q/2$.
\end{enumerate}
\item $\abs{x+d} \geq m$
\begin{enumerate}
\item $\abs{d}^{p} \leq q/2$,
\item $\abs{d}^{p} \geq q/2$.
\end{enumerate}
\end{enumerate}
Similarly as in Proposition~\ref{prop:concavity0}, only the cases 1a and 2a are delicate.

We continue to use the notation \eqref{eq:def:psi} and the estimates \eqref{eq:psi':size}, \eqref{eq:psi':H}.

\textbf{Case 1.}
First, we consider the case $\abs{x+d} \leq m$.

\textbf{Case 1a.}
We consider the subcase $u+e \leq v$, $\abs{d}^{p} \leq q/2$.

Let
\begin{align*}
G(t) &:= U(x+td,m,q+\abs{td}^{p},u+te,v)
\\ &=
\frac{u+te}{p'} (m^{p}/\CH^{p}+q+\abs{td}^{p})^{1/p} + \frac{u+te}{p} \frac{q+\abs{td}^{p}+\abs{x+td}^{p}/\CH^{p}}{(m^{p}/\CH^{p}+q+\abs{td}^{p})^{1-1/p}}
\\ &\quad-
v(q+\abs{td}^{p})^{1/p}(C-\tilde{C} \ln(1+(u+te)/v)).
\end{align*}
With this notation, the claim \eqref{eq:concavity} turns into $G(1) \leq G(0) + G'(0)$.

With a splitting $C=C_{1}+C_{2}$ to be chosen later, let
\begin{align*}
H(t) &:= 
\frac{u+te}{p'} (m^{p}/\CH^{p}+q)^{1/p} + \frac{u+te}{p} \frac{q+\psi(0)+\psi'(0)t}{(m^{p}/\CH^{p}+q)^{1-1/p}}
\\ &\quad-
C_{1}v(q+\abs{td}^{p})^{1/p} - v q^{1/p} (C_{2}-\tilde{C} \ln(1+(u+te)/v)).
\end{align*}
Then $G(0)=H(0)$ and $G'(0)=H'(0)$.
By \eqref{eq:1}, we have
\begin{multline*}
G(1)-H(1)
=
\frac{u+e}{p'} (m^{p}/\CH^{p}+q+\abs{d}^{p})^{1/p} + \frac{u+e}{p} \frac{q+\abs{d}^{p}+\abs{x+d}^{p}/\CH^{p}}{(m^{p}/\CH^{p}+q+\abs{d}^{p})^{1-1/p}}
\\ -
v(q+\abs{d}^{p})^{1/p}(C-\tilde{C} \ln(1+(u+e)/v))
-
\Bigl( \frac{u+e}{p'} (m^{p}/\CH^{p}+q)^{1/p} + \frac{u+e}{p} \frac{q+\psi(0)+\psi'(0)}{(m^{p}/\CH^{p}+q)^{1-1/p}}
\\ -
C_{1}v(q+\abs{d}^{p})^{1/p} - v q^{1/p} (C_{2}-\tilde{C} \ln(1+(u+e)/v)) \Bigr)
\\ \leq
\frac{u+e}{p'} (m^{p}/\CH^{p}+q+\abs{d}^{p})^{1/p} + \frac{u+e}{p} \frac{q+\abs{d}^{p}+\psi(0)+\psi'(0)+\abs{d}^{p}/p}{(m^{p}/\CH^{p}+q+\abs{d}^{p})^{1-1/p}}
\\ -
v(q+\abs{d}^{p})^{1/p}(C_{2}-\tilde{C} \ln(1+(u+e)/v))
-
\Bigl( \frac{u+e}{p'} (m^{p}/\CH^{p}+q)^{1/p} + \frac{u+e}{p} \frac{q+\psi(0)+\psi'(0)}{(m^{p}/\CH^{p}+q)^{1-1/p}}
\\ - v q^{1/p} (C_{2}-\tilde{C} \ln(1+(u+e)/v)) \Bigr)
=
K(\abs{d}^{p}) - K(0),
\end{multline*}
where
\begin{multline*}
K(s)
=
\frac{u+e}{p'} (m^{p}/\CH^{p}+q+s)^{1/p} + \frac{u+e}{p} \frac{q+s+a+b+s/p}{(m^{p}/\CH^{p}+q+s)^{1-1/p}}
\\-
v (q+s)^{1/p}(C_{2}-\tilde{C} \ln(1+(u+e)/v)).
\end{multline*}
In order to show that $G(1) \leq H(1)$, we compute
\begin{align*}
K'(s)
&=
\frac{u+e}{p'} \frac{1}{p(m^{p}/\CH^{p}+q+s)^{1-1/p}}
\\ &\quad+
\frac{u+e}{p} \Bigl( \frac{1+1/p}{(m^{p}/\CH^{p}+q+s)^{1-1/p}}
-
\frac{(1-1/p)(q+s+a+b+s/p)}{(m^{p}/\CH^{p}+q+s)^{2-1/p}} \Bigr)
\\ &\quad-
\frac{v}{p(q+s)^{1/p}}(C_{2}-\tilde{C} \ln(1+(u+e)/v))
\\ &\leq
(u+e) \frac{2/p}{(m^{p}/\CH^{p}+q+s)^{1-1/p}}
- \frac{v}{p(q+s)^{1/p}}(C_{2}-\tilde{C} \ln(2)).
\\ &\leq
\frac{2v}{p} \frac{1}{(q+s)^{1/p}}
- \frac{v}{p(q+s)^{1/p}}(C_{2}-\tilde{C} \ln(2))
\leq 0
\end{align*}
provided that
\begin{equation}
\label{eq:cond5}
C_{2}\geq 2 + \tilde{C} \ln(2).
\end{equation}

Now, it suffices to show $H(1) \leq H(0)+H'(0)$.
This will follow from $H'(t) \leq H'(0)$ for $t\in [0,1]$.
Compute
\begin{align*}
H'(t) &=
\frac{e}{p'} (m^{p}/\CH^{p}+q)^{1/p} + \frac{e}{p} \frac{q+a+bt}{(m^{p}/\CH^{p}+q)^{1-1/p}}
+
\frac{u+te}{p}\frac{b}{(m^{p}/\CH^{p}+q)^{1-1/p}}
\\ &\quad-
\frac{t^{p-1}\abs{d}^{p} C_{1}v}{(q+\abs{td}^{p})^{1/p}} + eq^{1/p} \tilde{C}/(1+(u+te)/v),
\end{align*}
and
\begin{align*}
H''(t) &=
\frac{2e}{p} \frac{b}{(m^{p}/\CH^{p}+q)^{1-1/p}}
\\ &\quad-
\frac{(p-1)t^{p-2}\abs{d}^{p} C_{1}v}{(q+\abs{td}^{p})^{1-1/p}}
+
\frac{(1-1/p)pt^{p-1}\abs{d}^{p} t^{p-1}\abs{d}^{p} C_{1}v}{(q+\abs{td}^{p})^{2-1/p}}
\\ &\quad- \frac{e^{2}}{v} q^{1/p}\tilde{C}/(1+(u+te)/v)^{2}
\\ &\leq
2\abs{e} \frac{a^{1/p'} \abs{d}}{(m^{p}/\CH^{p}+q)^{1-1/p}}
\\ &\quad-
\frac{(p-1) t^{p-2} \abs{d}^{p} C_{1}v}{(q+\abs{td}^{p})^{1-1/p}} \Bigl( 1 - \frac{\abs{td}^{p}}{q+\abs{td}^{p}} \Bigr)
- \frac{e^{2}}{v} q^{1/p}\tilde{C}/4
\\ &\leq
2 \abs{e} \abs{d}
-
\frac{(p-1) t^{p-2} \abs{d}^{p} C_{1}v}{(3/2)^{2-1/p} q^{1-1/p}}
- \frac{e^{2}}{v} q^{1/p}\tilde{C}/4.
\end{align*}
Integrating this inequality, we obtain
\[
H'(t) - H'(0)
\leq
2 \abs{e} \abs{d} t
- \frac{t^{p-1} \abs{d}^{p} C_{1}v}{(3/2)^{2-1/p} q^{1-1/p}}
- \frac{e^{2} q^{1/p}\tilde{C}}{4v} t.
\]
Recalling \eqref{eq:AMGM-ed}, we see that $H'(t) \leq H'(0)$ provided that
\begin{equation}
\label{eq:cond6}
C_{1} \geq 2 \cdot (3/2)^{2-1/p}/p,
\quad
\tilde{C} \geq 2^{3}/p'.
\end{equation}

\textbf{Case 1b.}
We keep the assumptions $\abs{x+d} \leq m$ and $\abs{d}^{p} \leq q/2$.
Now, we consider the case $u+e\geq v$.
In particular, $e\geq v-u \geq 0$.
Let
\begin{multline*}
J(e) :=
U\bigl(x+d, m , q+\abs{d}^{p}, u+e, u+e \bigr)
\\-
U(x,m,q,u,v)
-
U_{x}(x,m,q,u,v)d
-
U_{u}(x,m,q,u,v)e.
\end{multline*}
For $e = v-u$, we showed that $J(e)\leq 0$ in the previous case.
Hence, it suffices to show that $J'(e) \leq 0$ for $e \geq v-u$.
We have
\begin{align*}
J'(e) &=
\frac{1}{p'} (m^{p}/\CH^{p}+q+\abs{d}^{p})^{1/p}
+ \frac{1}{p} \frac{q+\abs{d}^{p}+\abs{x+d}^{p}/\CH^{p}}{(m^{p}/\CH^{p}+q+\abs{d}^{p})^{1-1/p}}
- (q+\abs{d}^{p})^{1/p} (C-\tilde{C}\ln 2)
\\ &\quad- \frac{1}{p'} (m^{p}/\CH^{p}+q)^{1/p}
- \frac{1}{p} \frac{q+\abs{x}^{p}/\CH^{p}}{(m^{p}/\CH^{p}+q)^{1-1/p}}
- \tilde{C} q^{1/p}/(1+u/v)
\\ &\leq
\frac{\abs{d}}{p'}
+ \frac{1}{p} \frac{\abs{d}^{p}+\abs{x+d}^{p}/\CH^{p}-\abs{x}^{p}/\CH^{p}}{(m^{p}/\CH^{p}+q)^{1-1/p}}
- (q+\abs{d}^{p})^{1/p} (C-\tilde{C}\ln 2).
\end{align*}
Using \eqref{eq:psi-linear-approx}, we obtain
\begin{align*}
J'(e) &\leq
\frac{\abs{d}}{p'}
+ \frac{1}{p} \frac{\abs{d}^{p} + \abs{b} + \abs{d}^{p}/p}{(m^{p}/\CH^{p}+q)^{1-1/p}}
- (q+\abs{d}^{p})^{1/p} (C-\tilde{C}\ln 2)
\\ &\leq
(1/p) \abs{d}
+ \frac{1}{p} \frac{\abs{d}^{p}(1+1/p) + p a^{p/p'} \abs{d}}{(m^{p}/\CH^{p}+q)^{1-1/p}}
- (q+\abs{d}^{p})^{1/p} (C-\tilde{C}\ln 2)
\\ &\leq
(1/p + (1+(1/p+1/p^{2})^{p})^{1/p}) \abs{d}
- (q+\abs{d}^{p})^{1/p} (C-\tilde{C}\ln 2).
\end{align*}
This is $\leq 0$ provided that
\begin{equation}
\label{eq:cond7}
C \geq 1/p + (1+(1/p+1/p^{2})^{p})^{1/p} + \tilde{C}\ln 2.
\end{equation}

\textbf{Case 1c.}
Suppose now that still $\abs{x+d} \leq m$, but now $\abs{d}^{p} \geq q/2$.
Let $\tilde{v} := (u+e) \vee v$.
Let
\begin{multline*}
I(t) := U(x+td,m,q+\abs{td}^{p},u+e,\tilde{v}) \\
- U(x,m,q,u,v) - U_{x}(x,m,q,u,v) td - U_{u}(x,m,q,u,v)e.
\end{multline*}
For $\abs{td}^{p} \leq q/2$, we showed $I(t) \leq 0$ in the previous steps.
Hence, it suffices to show $I'(t) \leq 0$ for $t\in [0,1]$ such that $\abs{td}^{p} \geq q/2$.
We have
\begin{multline*}
I'(t) =
\frac{u+e}{p'} \frac{t^{p-1}\abs{d}^{p}}{(m^{p}/\CH^{p}+q+\abs{td}^{p})^{1-1/p}}
\\ + \frac{u+e}{p} \Bigl( \frac{pt^{p-1}\abs{d}^{p}+\psi'(t)}{(m^{p}/\CH^{p}+q+\abs{td}^{p})^{1-1/p}}
- \frac{(q+\abs{td}^{p}+\abs{x+td}^{p}/\CH^{p}) \cdot (1-1/p)pt^{p-1} \abs{d}^{p}}{(m^{p}/\CH^{p}+q+\abs{td}^{p})^{2-1/p}} \Bigr)
\\-\frac{\tilde{v} t^{p-1} \abs{d}^{p}}{(q+\abs{td}^{p})^{1-1/p}}(C - \tilde{C}\ln(1+(u+e)/\tilde{v}))
- \frac{u \psi'(0)}{p(m^{p}/\CH^{p}+q)^{1-1/p}}.
\end{multline*}
\[
\leq
\tilde{v} \Bigl( \frac{(1+1/p') t^{p-1} \abs{d}^{p} + \psi(t)^{1/p'} \abs{d}}{(m^{p}/\CH^{p}+q+\abs{td}^{p})^{1-1/p}} \Bigr)
-\frac{\tilde{v} t^{p-1} \abs{d}^{p}}{(q+\abs{td}^{p})^{1-1/p}}(C - \tilde{C}\ln2)
+ \frac{\tilde{v} a^{1/p'} \abs{d}}{(m^{p}/\CH^{p}+q)^{1-1/p}}
\]
\[
\leq
\tilde{v} \abs{d} ((1+1/p')^{p}+1)^{1/p}
-\frac{\tilde{v} t}{3^{1-1/p}}(C - \tilde{C}\ln2) + \tilde{v} \abs{d}
\leq 0
\]
provided that
\begin{equation}
\label{eq:cond8}
C \geq 3^{1-1/p}(1+(1+(1/p')^{p})^{1/p}) + \tilde{C} \ln 2.
\end{equation}

\textbf{Case 2.}
Suppose now that $\abs{x+d} > m$.
Let $\tilde{v} := v \vee (u+e)$.
In this case, the claim \eqref{eq:concavity} becomes
\begin{multline}
\label{eq:claim-m-changes}
(u+e)(\abs{x+d}^{p}/\CH^{p}+q+\abs{d}^{p})^{1/p}
-C\tilde{v}(q+\abs{d}^{p})^{1/p} + \tilde{C} \tilde{v} (q+\abs{d}^{p})^{1/p} \ln(1+(u+e)/\tilde{v})
\\ \leq
\frac{u}{p'} (m^{p}/\CH^{p}+q)^{1/p}+ \frac{u}{p} \frac{q+\abs{x}^{p}/\CH^{p}}{(m^{p}/\CH^{p}+q)^{1-1/p}}
-Cv q^{1/p} + \tilde{C} v q^{1/p} \ln(1+u/v)
\\ +
\frac{u\psi'(0)}{p(m^{p}/\CH^{p}+q)^{1-1/p}} + \Bigl(\frac{1}{p'}(m^{p}/\CH^{p}+q)^{1/p} + \frac{1}{p}\frac{q+\abs{x}^{p}/\CH^{p}}{(m^{p}/\CH^{p}+q)^{1-1/p}}
+ \tilde{C} q^{1/p}/(1+u/v)\Bigr) e.
\end{multline}

\textbf{Case 2a.}
Suppose first $\abs{d}^{p} \leq q/2$.
Let a splitting $C=C_{3}+C_{4}$ be chosen later.
By Proposition~\ref{prop:concavity0}, we have
\begin{multline*}
(u+e)(\abs{x+d}^{p}/\CH^{p}+q+\abs{d}^{p})^{1/p}
-C_{3}\tilde{v}(q+\abs{d}^{p})^{1/p}
+ \tilde{C} \tilde{v} (q+\abs{d}^{p})^{1/p} \ln(1+(u+e)/\tilde{v})
\\ \leq
u(\abs{x}^{p}/\CH^{p}+q)^{1/p}
-C_{3} v q^{1/p} + \tilde{C} v q^{1/p} \ln(1+u/v)
\\ +
\frac{u \psi'(0)}{p(\abs{x}^{p}/\CH^{p}+q)^{1/p}} + \Bigl((\abs{x}^{p}/\CH^{p}+q)^{1/p} + \tilde{C} q^{1/p}/(1+u/v)\Bigr) e
\end{multline*}
with 
\begin{equation}
\label{eq:cond9}
C_{3} = 9, \quad \tilde{C} = 4\sqrt{2}.
\end{equation}
Note that this value of $\tilde{C}$ is compatible with \eqref{eq:cond6}.

By the AMGM inequality, we have
\[
(\abs{x}^{p}/\CH^{p}+q)^{1/p}
\leq
\frac{1}{p'} (m^{p}+q)^{1/p} + \frac{1}{p} \frac{q+\abs{x}^{p}/\CH^{p}}{(m^{p}+q)^{1-1/p}}.
\]
Multiplying this inequality by $u+e$ and inserting it on the right-hand side of \eqref{eq:claim-m-changes}, we see that it suffices to show
\[
C_{4}v q^{1/p} - C_{4} \tilde{v} (q+\abs{d}^{p})^{1/p}
\leq
\frac{u \psi'(0)}{p(m^{p}/\CH^{p}+q)^{1-1/p}}
- \frac{u \psi'(0)}{p(\abs{x}^{p}/\CH^{p}+q)^{1-1/p}}.
\]
If $\psi'(0) \leq 0$, then the right-hand side is positive, so there is nothing to show.
Let us assume $\psi'(0)> 0$.
The claim is then equivalent to
\[
\frac{u \psi'(0)}{p} \bigl( \frac{1}{(\abs{x}^{p}/\CH^{p}+q)^{1-1/p}}
- \frac{1}{(m^{p}/\CH^{p}+q)^{1-1/p}} \bigr)
\leq
C_{4} \tilde{v} (q+\abs{d}^{p})^{1/p} - C_{4}v q^{1/p}.
\]
Since $0 \leq u \leq v \leq \tilde{v}$ and $\abs{\psi'(0)} \leq p a^{1/p'}\abs{d}$, it suffices to show
\[
\frac{\abs{x}^{p/p'}\abs{d}}{\CH^{p}} \bigl(\frac{1}{(\abs{x}^{p}/\CH^{p}+q)^{1-1/p}} -\frac{1}{(m^{p}/\CH^{p}+q)^{1-1/p}} \bigr)
\leq
C_{4} (q+\abs{d}^{p})^{1/p} - C_{4} q^{1/p}.
\]
By concavity of $\cdot^{1/p}$ and convexity of $\cdot^{1/p-1}$, this will follow from
\[
\frac{\abs{x}^{p/p'}\abs{d}}{\CH^{p}}
(\abs{x}^{p}/\CH^{p}-m^{p}/\CH^{p}) (1/p-1) (\abs{x}^{p}/\CH^{p}+q)^{1/p-2}
\leq
C_{4} \abs{d}^{p} (1/p) (q+\abs{d}^{p})^{1/p-1}.
\]
Since $\abs{d}^{p} \leq q/2$, this will follow from
\[
\frac{\abs{x}^{p/p'}\abs{d}}{\CH^{2p}}
(p-1) (m^{p}-\abs{x}^{p}) (\abs{x}^{p}/\CH^{p}+q)^{1/p-2}
\leq
C_{4} \abs{d}^{p} (3q/2)^{1/p-1}.
\]
This will follow from
\[
\frac{\abs{d}}{\CH^{p}}
(p-1) (m^{p}-\abs{x}^{p}) (\abs{x}^{p}/\CH^{p}+q)^{-1}
\leq
C_{4} \abs{d}^{p} (3q/2)^{1/p-1}.
\]
Since $\abs{x} \leq m \leq \abs{x+d} \leq \abs{x} + \abs{d}$ and using the elementary inequality $(a+b)^{p} \leq a+pa^{p-1}b+b^{p}$ (which holds for any $a,b\geq 0$ and $p \in [1,2]$, as can be seen by differentiating both sides in $b$), we have
\begin{multline*}
m^{p}/\CH^{p}+q
\leq
(\abs{x} + \abs{d})^{p}/\CH^{p}+q
\leq
(\abs{x}^{p} + p\abs{x}^{p-1}\abs{d} + \abs{d}^{p})/\CH^{p}+q
\\ \leq
\abs{x}^{p}/\CH^{p} + p(\abs{x}/\CH)^{p-1}\abs{d} + 2q
\leq
\abs{x}^{p}/\CH^{p} + p((\abs{x}/\CH)^{p}/p'+\abs{d}^{p}/p) + 2q
\\ \leq
3(\abs{x}^{p}/\CH^{p} + q),
\end{multline*}
so it suffices to show
\[
3(p-1)\frac{\abs{d}}{\CH^{p}}
(m^{p}-\abs{x}^{p})
\leq
C_{4} \abs{d}^{p} (3q/2)^{1/p-1} (m^{p}/\CH^{p}+q).
\]
\[
\impliedby
3(p-1)\frac{\abs{d}}{\CH^{p}}
p (m-\abs{x})m^{p-1}
\leq
C_{4} \abs{d}^{p} (3q/2)^{1/p-1} (m^{p}/\CH^{p}+q).
\]
\[
\impliedby
3(p-1)\frac{\abs{d}}{\CH^{p}}
p (m-\abs{x})m^{p-1}
\leq
C_{4} \abs{d}^{p} (3/2)^{1/p-1} (m^{p}/\CH^{p}+q)^{1/p}.
\]
\[
\impliedby
3(p-1)\abs{d}
p (m-\abs{x})m^{p-1}
\leq
C_{4} \abs{d}^{p} (3/2)^{1/p-1} m.
\]
Since $m-\abs{x} \leq \min(\abs{d},m)$, this holds provided that
\begin{equation}
\label{eq:cond10}
C_{4} \geq 3(p-1) p (3/2)^{1-1/p}.
\end{equation}

\textbf{Case 2b.}
Suppose now $\abs{d}^{p}\geq q/2$.
Let
\begin{multline*}
I(t) := (u+e)(\abs{x+td}^{p}/\CH^{p}+q+\abs{td}^{p})^{1/p}
-C\tilde{v}(q+\abs{td}^{p})^{1/p} + \tilde{C} \tilde{v} (q+\abs{td}^{p})^{1/p} \ln(1+(u+e)/\tilde{v})
\\ -
\frac{u}{p'} (m^{p}/\CH^{p}+q)^{1/p}
- \frac{u}{p} \frac{q+\abs{x}^{p}/\CH^{p}}{(m^{p}/\CH^{p}+q)^{1-1/p}}
+ Cv q^{1/p} - \tilde{C} v q^{1/p} \ln(1+u/v)
\\ -
\frac{u\psi'(0)t}{p(m^{p}/\CH^{p}+q)^{1-1/p}}
- \Bigl(\frac{1}{p'}(m^{p}/\CH^{p}+q)^{1/p} + \frac{1}{p} \frac{q+\abs{x}^{p}/\CH^{p}}{(m^{p}/\CH^{p}+q)^{1-1/p}}
+ \tilde{C} q^{1/p}/(1+u/v)\Bigr) e.
\end{multline*}
The claim \eqref{eq:claim-m-changes} is equivalent to $I(1) \leq 0$.
From the previous cases, we know that $I(t) \leq 0$ if $t$ is so small that either $\abs{x+td}=m$ or $\abs{td}^{p} \leq q/2$.
Hence, it suffices to show $I'(t) \leq 0$ for all $t \in [0,1]$ such that $\abs{x+td}\geq m$ and $\abs{td}^{p} \geq q/2$.
We compute
\begin{align*}
I'(t)
&=
(u+e) \frac{\psi'(t)+pt^{p-1}\abs{d}^{p}}{p(\abs{x+td}^{p}/\CH^{p}+q+\abs{td}^{p})^{1-1/p}}
\\ &\quad -\frac{\tilde{v} t^{p-1} \abs{d}^{p}}{(q+\abs{td}^{p})^{1-1/p}}(C - \tilde{C}\ln(1+(u+e)/\tilde{v}))
- \frac{u \psi'(0)}{p (m^{p}/\CH^{p}+q)^{1-1/p}}
\\ &\leq
(u+e) \abs{d} \frac{\abs{x+td}^{p/p'}/\CH^{p}+\abs{td}^{p-1}}{(\abs{x+td}^{p}/\CH^{p}+q+\abs{td}^{p})^{1-1/p}}
\\ &\quad
-\frac{\tilde{v} t^{p-1} \abs{d}^{p}}{(q+\abs{td}^{p})^{1-1/p}}(C - \tilde{C}\ln(2))
+ \frac{u \abs{x}^{p/p'} \abs{d}/\CH^{p}}{(m^{p}/\CH^{p}+q)^{1-1/p}}
\\ &\leq
(u+e) \abs{d} 2^{1/p}
-\frac{\tilde{v} \abs{d}}{3^{1-1/p}}(C - \tilde{C}\ln(2)) + u \abs{d}.
\end{align*}
This is negative provided that
\begin{equation}
\label{eq:cond11}
C \geq 3^{1-1/p}(1+2^{1/p}) + \tilde{C}\ln(2).
\end{equation}

\textbf{Conclusion of the proof.}
The conditions \eqref{eq:cond5}, \eqref{eq:cond6}, \eqref{eq:cond7}, \eqref{eq:cond8}, \eqref{eq:cond9}, \eqref{eq:cond10}, \eqref{eq:cond11} amount to $\tilde{C} = 4\sqrt{2}$ and
\begin{multline}
\label{eq:cond-C}
C \geq \max(2 \cdot (3/2)^{2-1/p}/p + 2, 1/p + (1+(1/p+1/p^{2})^{p})^{1/p},\\ 3^{1-1/p}(1+(1+(1/p')^{p})^{1/p}), 9+3(p-1) p (3/2)^{1-1/p}, \\
3^{1-1/p}(1+2^{1/p}) ) + \tilde{C}\ln(2).
\end{multline}
The latter condition holds for $C=21$.
\end{proof}

\appendix
\section{Extrapolation}
\label{sec:extrapolation}
Here, we show how weighted estimates such as \eqref{eq:w-BDG} can be used to obtain a different kind of vector-valued estimates,
where the vector space is a UMD Banach function space.
We recall that a Banach space $X$ is UMD if and only if every martingale transform with bounded coefficients on a filtered probability space $\Omega$ defines a bounded operator on every $L^{r}(\Omega,X)$ with $r \in (1,\infty)$; we refer to \cite[Section 4]{MR3617205} for many equivalent characterizations (including via the eponymous unconditional summability of martingale differences) and examples of UMD spaces.
We start with a very simple extrapolation result, which uses only the duality argument introduced in \cite{MR0284802}.
 
\begin{proposition}[Banach function space valued extrapolation]
\label{prop:UMD-extrapolation}
For every $r\in (1,\infty)$ and every UMD Banach function space $X$ over a $\sigma$-finite measure space $(S,\Sigma,\mu)$, there exists a constant $C_{r,X}<\infty$ such that the following holds.

Let $(\Omega,(\calF_{n})_{n\in\N})$ be a filtered probability space and let $f,g : \Omega \times S \to \R_{\geq 0}$ be measurable functions such that, for some $A<\infty$, $\mu$-almost every $s\in S$, and every weight $w$ on $\Omega$, we have
\begin{equation}
\label{eq:UMD-extrapolation:A1}
\int_{\Omega} f(\cdot,s) w
\leq
A \int_{\Omega} g(\cdot,s) w^{*}.
\end{equation}
Then, we have
\begin{equation}
\label{eq:UMD-extrapolation:Lp}
\norm{f}_{L^{r}(\Omega,X)} \leq C_{r,X} A \norm{g}_{L^{r}(\Omega,X)}.
\end{equation}
\end{proposition}

Proposition~\ref{prop:UMD-extrapolation} is not new, in the sense that it also follows from the usual Rubio de Francia extrapolation argument and (a martingale version of) \cite[Theorem 2.4.1]{arxiv:1912.09347}.
However, the direct proof is substantially easier.

\begin{proof}
Truncating $f$, we may assume that the left-hand side of \eqref{eq:UMD-extrapolation:Lp} is finite.

Recall that the Banach function space $X$ has an \emph{associate Banach function space} $X'$, which is again a Banach function space on $(S,\Sigma,\mu)$ with the property that, for every $h\in X'$, we have
\[
\norm{h}_{X'} = \sup_{f\in X, \norm{f}_{X}\leq 1} \int_{S} fh \dif\mu,
\]
and all these integrals converge absolutely.
In particular, $X'$ is isomorphic to a closed subspace of the dual space of $X$.
The associate space is \emph{norming} in the sense that, for every $f\in X$, we have
\[
\norm{f}_{X} = \sup_{h\in X', \norm{h}_{X'}\leq 1} \int_{S} fh \dif\mu,
\]
see e.g.\ \cite[\textsection 71]{MR0222234} for the proof of this fact.

Next, we need a measurable selection of $h(f)$ that almost extremize the supremum on the right-hand side.
In concrete spaces $X$, it is frequently possible to explicitly find such a selection.
For an abstract Banach function space $X$, it seems necessary to reduce to a finite-dimensional subspace first.

Since $X$ is UMD, it is reflexive \cite[Theorem 4.3.3]{MR3617205}.
It follows that the norm of $X$ is absolutely continuous \cite[\textsection 73, Theorem 2]{MR0222234}.
Therefore, a version of the dominated convergence theorem holds in $X$ \cite[\textsection 72, Theorem 2]{MR0222234}.
Hence, we can approximate $f$ by simple functions in $L^{r}(\Omega,X)$, that is, finite linear combinations of characteristic functions of product subsets of $\Omega \times S$.

Let $\epsilon>0$, and let $\mathfrak{f}$ be such a simple function with
\[
\norm{f - \mathfrak{f}}_{L^{r}(\Omega,X)} < \epsilon.
\]
Then $\mathfrak{f}(\omega,\cdot)$ takes values in a finite-dimensional subspace of $X$ as $\omega\in\Omega$ varies.
On this finite-dimensional subspace, we may choose a measurable map $f \mapsto h(f)$ such that
\begin{equation}
\label{eq:h-norm}
\norm{h(f)}_{X'} \leq \norm{f}_{X}^{r-1},
\quad\text{and}
\end{equation}
\begin{equation}
\label{eq:h-alm-extremizer}
(1+\epsilon) \int_{S} f h(f) \dif \mu \geq \norm{f}_{X}^{r}.
\end{equation}
We do this by first choosing such a map on the unit sphere, using compactness of the unit sphere, and then extend it by homogeneity.

By the hypothesis \eqref{eq:UMD-extrapolation:A1} with the weights
\[
w(\omega,s) := h(\mathfrak{f}(\omega,\cdot))(s),
\]
we obtain
\begin{equation}
\label{eq:f-dualized}
\int_{\Omega} \int_{S} f(\omega,s) w(\omega,s) \dif\mu(s) \dif\omega
\leq
A \int_{S} \int_{\Omega} g(\omega,s) w^{*}(\omega,s) \dif\mu(s) \dif\omega,
\end{equation}
where $w^{*}(\omega,s) := \sup_{n\geq 0} \E(w(\cdot,s)|\calF_{n})(\omega)$.
By duality,
\[
\eqref{eq:f-dualized} \leq
A \int_{\Omega} \norm{g(\omega,\cdot)}_{X} \norm{w^{*}(\omega,\cdot)}_{X'} \dif\omega
\leq
A \norm{g}_{L^{r}(\Omega,X)} \norm{w^{*}}_{L^{r'}(\Omega,X')}.
\]
Since $X'$ is a closed subspace of the Banach space dual of $X$, it is again UMD \cite[Proposition 4.2.17]{MR3617205}.
By the UMD-valued maximal inequality \cite[Theorem 3.2]{MR4181372}, we have
\[
\norm{w^{*}}_{L^{r'}(\Omega,X')} \leq C_{r',X'} \norm{w}_{L^{r'}(\Omega,X')}.
\]
By \eqref{eq:h-norm}, we have
\begin{equation}
\label{eq:w-norm}
\norm{w}_{L^{r'}(\Omega,X')} \leq \norm{\frakf}_{L^{r}(\Omega,X)}^{r-1}.
\end{equation}
It follows that
\[
\eqref{eq:f-dualized}
\leq
A C_{r',X'} \norm{g}_{L^{r}(\Omega,X)} \norm{\frakf}_{L^{r}(\Omega,X)}^{r-1}.
\]
Finally, by duality and \eqref{eq:w-norm}, we have
\[
\int_{\Omega} \int_{S} (f-\mathfrak{f})(\omega,s) w(\omega,s) \dif\mu(s) \dif\omega
\leq
\norm{f-\mathfrak{f}}_{L^{r}(\Omega,X)} \norm{\frakf}_{L^{r}(\Omega,X)}^{r-1}.
\]
Combining the last two inequalities, we obtain
\begin{align*}
\norm{\mathfrak{f}}_{L^{r}(\Omega,X)}^{r}
&\leq
(1+\epsilon) \int_{\Omega} \int_{S} \mathfrak{f}(\omega,s) h(\mathfrak{f}(\omega,\cdot))(s) \dif\mu(s) \dif\omega
\\ &\leq
(1+\epsilon) (A C_{r',X'} \norm{g}_{L^{r}(\Omega,X)} + \norm{f-\mathfrak{f}}_{L^{r}(\Omega,X)}) \norm{\frakf}_{L^{r}(\Omega,X)}^{r-1}.
\end{align*}
Since $\epsilon$ can be chosen arbitrarily small, this implies the claim \eqref{eq:UMD-extrapolation:Lp}.
\end{proof}

\begin{remark}
The space $L^{r}(\Omega)$ in Proposition~\ref{prop:UMD-extrapolation} can be replaced by another Banach function space $Y$, provided that $Y'(X')$ is a norming subspace of the dual space of $Y(X)$, and, most importantly, that the martingale maximal operator is bounded on $Y'(X')$.
One example is when $Y$ is a weighted $L^{r}$ space and $X=\R$; the appropriate maximal bounds in this case have been proved in \cite{MR3916937}.
\end{remark}

As a direct consequence of the weighted BDG inequality \eqref{eq:w-BDG} (or rather just the scalar case from \cite{MR3688518}) and Proposition~\ref{prop:UMD-extrapolation}, we recover the following BDG-type inequality.
\begin{corollary}[{\cite[Theorem 1.1]{MR4181372}}]
\label{thm:UMD-BFS-BDG}
Let $r,S,X,C_{r,X}$ be as in Proposition~\ref{prop:UMD-extrapolation}.
Let $(\Omega,(\calF_{n})_{n\in\N})$ be a filtered probability space.
Let $f : \N \times \Omega \times S \to \R$ be a function with $f_{0}(\omega,s)=0$ such that
\begin{enumerate}
\item for every $n$, $f_{n}(\cdot,\cdot)$ is $\calF_{n}\times\Sigma$-measurable, and
\item for almost every $s\in S$, $(f_{n}(\cdot,s))_{n\in\N}$ is a martingale with respect to $(\calF_{n})_{n}$.
\end{enumerate}
Then,
\begin{equation}
\label{eq:UMD-BFS-BDG}
\norm{ \sup_{n\in\N} \abs{f_{n}(\cdot,\cdot)} }_{L^{r}(\Omega,X)}
\leq 16 (\sqrt{2}+1) C_{r,X}
\norm[\Big]{ \bigl( \sum_{n\geq 1} \abs{f_{n}(\cdot,\cdot)-f_{n-1}(\cdot,\cdot)}^{2} \bigr)^{1/2} }_{L^{r}(\Omega,X)}.
\end{equation}
\end{corollary}

\begin{proof}
By the monotone convergence theorem, we may consider a finite sequence of times $n\leq N$, so that the left-hand side of \eqref{eq:UMD-BFS-BDG} is finite if the right-hand side is.
Let
\[
f(\omega,s) := \max_{n\leq N} \abs{f_{n}(\omega,s)}.
\]
This is an $\calF_{N}\times\Sigma$-measurable function, and for a.e.\ $\omega$ we have $f(\omega,\cdot) \in X$.
By \cite[Theorem 1.1]{MR3688518}, the hypothesis \eqref{eq:UMD-extrapolation:A1} of Proposition~\ref{prop:UMD-extrapolation} holds for the above function $f$ with $A=16(\sqrt{2}+1)$ and
\[
g(\omega,s) = \bigl( \sum_{n=1}^{N} \abs{f_{n}(\omega,s)-f_{n-1}(\omega,s)}^{2} \bigr)^{1/2}.
\qedhere
\]
\end{proof}

\begin{remark}
In \cite[Theorem 1.1]{MR4181372}, also a converse inequality to \eqref{eq:UMD-BFS-BDG} has been proved.
That converse inequality does not follow from the main result of \cite{MR3567926}, due to the restriction to weights that are almost surely continuous in time in that result.
In \cite{arxiv:2106.11279}, we extend the main result of \cite{MR3567926} in such a way that it recovers the converse to \eqref{eq:UMD-BFS-BDG}.
\end{remark}

\paragraph*{Acknowledgment}
I thank Mark Veraar for making me aware of the recent developments concerning vector-valued extensions of Burkholder--Davis--Gundy inequalities.
I thank the anonymous referee for helpful improvement suggestions.

\printbibliography
\end{document}